\documentclass{amsart}
\usepackage{amsmath}
\usepackage{amssymb}
\usepackage{amsfonts}
\usepackage{color}
\newtheorem{theorem}{Theorem}[section]
\newtheorem{theor}{Theorem}
\newtheorem{teo}[theor]{Theorem}
\newtheorem{lm}[theorem]{Lemma}

\newtheorem{cor}[theorem]{Corollary}

\newtheorem{pr}[theorem]{Proposition}

\begin{document}
\title{Irreducibility of induced modules for general linear supergroups}
\author{Frantisek Marko}
\address{Penn State Hazleton, 76 University Drive, Hazleton, PA 18202, USA}
\email{fxm13@psu.edu}

\begin{abstract}
In this note we determine when is an induced module $H^0_G(\lambda)$,
corresponding to a dominant integral highest weight $\lambda$ of the general
linear supergroup $G=GL(m|n)$ irreducible. Using the contravariant duality
given by the supertrace we obtain a characterization of irreducibility of
Weyl modules $V(\lambda)$. This extends the result of Kac (\cite{kac2}, \cite%
{kac3}) who proved that, for ground fields of characteristic zero, $%
V(\lambda)$ is irreducible if and only if $\lambda$ is typical.
\end{abstract}

\maketitle
\section*{Introduction}

Induced modules (sometimes called Schur or costandard modules) play an
important role in the classical representation theory of Schur algebras and
general linear groups. Equally important are Weyl modules (sometimes called
standard modules), which are universal highest weight modules. The Weyl
modules are dual to the induced modules via the contravariant duality
induced by the trace. The induced modules and Weyl modules are building
blocks of the quasi-hereditary structure of the general linear groups. Since
the induced modules have simple socle (and the Weyl modules have simple top)
and all simple modules are obtained this way, it is natural to ask when are
the induced (and Weyl) modules irreducible. James and Mathas (Theorem 5.42
of \cite{mat}) proved a criterion for irreducibility of Weyl modules and
formulated a conjecture for irreducibility of Specht modules (Conjecture
5.47 of \cite{mat}) that play a dominant role in the representation theory
of symmetric groups. This conjecture was proved recently by Fayers in \cite%
{fa1, fa2}.

The classical Lie superalgebras were considered by Kac in \cite{kac}. Kac
also investigated the universal highest weight modules $V(\lambda)$ over a
ground field $K$ of characteristic zero in \cite{kac2, kac3}, and proved
that $V(\lambda)$ is irreducible if and only if $\lambda$ is typical. This
was in sharp contrast to the case of the general linear Lie algebra, in the
case of characteristic zero, when all modules are completely reducible.

The purpose of this paper is to characterize irreducibility of the induced
and Weyl modules for general linear supergroups in the case of arbitrary
characteristic.

Let $G=GL(m|n)$ be a general linear supergroup over a field $K$ of
characteristic different from $2$, and let $G_{ev}$ be the even
subsupergroup of $G$ isomorphic to $GL(m)\times GL(n)$. Denote by $%
\lambda=(\lambda^+|\lambda^-)=(\lambda^+_1, \ldots,\lambda^+_m
|\lambda^-_{1}, \ldots, \lambda^-_{n})$ a dominant integral weight of $G$,
that is $\lambda^+_1\geq \ldots \geq \lambda^+_m$ and $\lambda^-_{1} \geq
\ldots \lambda^-_{n}$. Let $B$ be the Borel subsupergroup of $G$
corresponding to lower triangular matrices and $K_{\lambda}$ be the
one-dimensional (even) $B$-supermodule corresponding to the weight $\lambda$%
. Corresponding to the embedding $B\to G$ there is an induction functor $%
ind_B^G: B-Mod \to G-Mod$ (See section 6 of \cite{bkl}.) 
The induced module $ind_B^G(K_{\lambda})$
will be denoted by $H^0_G(\lambda)$, and analogously, the induced module $%
ind_{B\cap G_{ev}}^{G_{ev}}(K_{\lambda})$ will be denoted by $%
H^0_{G_{ev}}(\lambda)$.

Define $\omega_{ij}(\lambda)=\lambda^+_i+\lambda^-_j+m+1-i-j$ and call the
weight $\lambda$ \emph{atypical} if $\omega_{ij}$ is a multiple of the characteristic 
of the ground field $K$ for some $i=1, \ldots, m$
and $j=1,\ldots, n$; otherwise call the weight $\lambda$ \emph{typical}.

\begin{teo}
\label{0} The $G$-module $H^0_G(\lambda)$ is irreducible if and only if $%
H^0_{G_{ev}}(\lambda)$ is an irreducible $G_{ev}$-module and $\lambda$ is
typical.
\end{teo}

One can view the above theorem as a first step towards the linkage principle
for the general linear supergroup in the case of positive characteristic.

Proposition 1.2 of \cite{fa1} gives combinatorial characterizations when
Weyl modules over the general linear group $GL(m)$ or $GL(n)$ are
irreducible. Combining them together and dualizing gives the combinatorial
criterion for the irreducibility of the induced $G_{ev}$-module $%
H_{G_{ev}}(\lambda )$.

The corresponding result for the Weyl module $V(\lambda)$ is obtained by
taking the contravariant dual. In the case when $charK=0$, we obtain a
different proof of the result of Kac mentioned earlier and we also get a
converse of the factorization formula (Theorem 6.20 of \cite{br}) of Berele
and Regev for hook Schur functions.

\section{Definitions and notation}

We start by reviewing definitions and certain results related to the general
linear supergroup $G=GL(m|n)$. Good references for our purposes are \cite{z}%
, \cite{bk} and \cite{zs}.

Let us write a generic $(m+n)\times (m+n)$-matrix $C$ as a block matrix 
\begin{equation*}
C=%
\begin{pmatrix}
C_{11} & C_{12} \\ 
C_{21} & C_{22}%
\end{pmatrix}%
,
\end{equation*}
where $C_{11}, C_{12}, C_{21}$ and $C_{22}$ are matrices of sizes $m\times m$%
, $n\times m$, $m\times n$ and $n\times n$, respectively. Denote by $%
S(C_{12})$ the symmetric superalgebra on coefficients of the matrix $C_{12}$%
, and denote $D= det(C_{11})$ and $D_{22}= det(C_{22})$. The commutative
superalgebra $A(m|n)$ is freely generated by elements $c_{ij} $ for $1\leq
i,j \leq m+n$. The coordinate ring $K[G]$ of $G$ is defined as a
localization of $A(m|n)$ at $det(C)=DD_{22}$. The general linear supergroup $%
G=GL(m|n)$ is a functor $Hom_{sup}(K[G],-)$ of morphisms preserving parity
of elements on commutative superalgebras. In this notation the weight $%
\lambda$ of $G$ is written as $\lambda=(\lambda_1, \ldots,
\lambda_m|\lambda_{m+1}, \ldots, \lambda_{m+n})$. We will identify $\lambda$
with our previous notation $\lambda=(\lambda^+|\lambda^-)$ and use both
notations interchangedly.

We now describe an explicit basis for $H^0_{G_{ev}}(\lambda)$ consisting of
bideterminants.

Assume first that $\lambda$ is a polynomial weight of $G$, that is $%
\lambda^+_m\geq 0$ and $\lambda^-_n\geq 0$. Define a tableaux $T_{\lambda}^+$
of the shape $\lambda^+$ as $T_{\lambda}^+(i,j)=i$ for $i=1, \ldots, m$ and $%
j=1,\ldots, \lambda_i$ and a tableaux $T_{\lambda}^-$ of the shape $\lambda^-
$ as $T_{\lambda}^-(i,j)=i$ for $i=m+1, \ldots, m+n$ and $j=1,\ldots,
\lambda_i$. Denote by $B^+(I)=(T_{\lambda}^+:T(I))$ and $B^-(J)=(T_{%
\lambda}^-:T(J))$ bideterminants corresponding to tableaux $T(I)$ and $T(J)$
of shape $\lambda^+$ and $\lambda^-$, respectively. Then the set of
bideterminants $B^+(I)$ for standard tableaux $T(I)$ forms a basis of the
simple $GL(m)$-module of the highest weight $\lambda^+$ and the set of
bideterminants $B^-(J)$ for standard tableaux $T(J)$ forms a basis of the
simple $GL(n)$-module of the highest weight $\lambda^-$. For more about
bideterminants, see \cite{m}. The induced module $H^0_{G_{ev}}(\lambda)$ has
a basis consisting of products of bideterminant $B^+(I)$ for $I$ standard of
shape $\lambda^+$ and of bideterminants $B^-(J)$ for $J$ standard of shape $%
\lambda^-$.

Now consider the case when $\lambda$ is not polynomial. If $\lambda^+_m<0$,
then the induced $GL(m)$-module $H^0_{GL(m)}(\lambda^+)$ is isomorphic to $%
H^0_{GL(m)}(\lambda^{++})\otimes (D^{\lambda^+_m})$, where the weight $%
\lambda^{++}=\lambda^+ -\lambda^+_m$ is polynomial and $(D^{\lambda^+_m})$
is a one-dimensional $GL(m)$-representation generated by $D^{\lambda^+_m}$.
Therefore the module $H^0_{GL(m)}(\lambda^+)$ has a basis consisting of
products of bideterminants $B^+(I)$ for $I$ standard of shape $\lambda^{++}$
multiplied by $D^{\lambda^+_m}$.

In the supercase, there is a group-like element 
\begin{equation*}
Ber(C)=det(C_{11}-C_{12}C_{22}^{-1}C_{21})det(C_{22})^{-1}
\end{equation*}
which generates an irreducible one-dimensional $G$-module $B$ of the weight $%
\beta=(-1,-1,\ldots, -1|1,1,\ldots, 1)$. Since $H^0_G(\lambda)\cong
H^0_G(\lambda-\lambda^-_n\beta) \otimes B^{\lambda^-_n}$, in what follows we
can assume that $\lambda^-$ is a polynomial weight of $GL(n)$, that is $%
\lambda^-_n\geq 0$, since we can reduce the general case to this one by
tensoring with $B^{\lambda^-_n}$.

Assuming $\lambda^-_n\geq 0$, the module $H^0_{G_{ev}}(\lambda)$ has a basis
that is a product of $D^{\lambda^+_m}$, bideterminants $B^+(I)$ for $I$
standard of shape $\lambda^{++}$ and bideterminants $B^-(J)$ for $J$
standard of shape $\lambda^-$. Denote $v_{ev}^{++}=(T_{\lambda^{++}}:T_{%
\lambda^{++}})$, $v_{ev}^+=v_{ev}^{++}D^{\lambda^+_m}$ and $%
v_{ev}^-=(T_{\lambda}^-:T_{\lambda}^-)$. The element $v_{ev}=v_{ev}^+v_{ev}^-
$ is the highest weight vector in $H^0_{G_{ev}}(\lambda)$ with respect 
to the dominance order corresponding to $B\cap G_{ev}$.

The $G$-supermodule $H^0_G(\lambda)$ is described explicitly using the
isomorphism $\phi :H^0_{G_{ev}}(\lambda)\otimes S(C_{12})\to H^0_G(\lambda)$
defined in Lemma 5.2 of \cite{z}. The map $\phi$ is given on generators as
follows: 
\begin{equation*}
C_{11}\mapsto C_{11}, C_{21}\mapsto C_{21}, C_{12}\mapsto
-C_{11}^{-1}C_{12}, C_{22}\mapsto C_{22}-C_{21}C_{11}^{-1}C_{12}.
\end{equation*}

In order to study the structure of the induced module $H^0_G(\lambda)$, and
of the simple $G$-module $L_G(\lambda)=\langle v \rangle$ of highest weight $%
\lambda$, we will use the superalgebra of distributions $Dist(G)$ of $G$
described in Section 3 of \cite{bk}. Denote $Dist_1(G)=(K[G]/\mathfrak{m}%
^2)^*$, where $*$ is the duality $Hom_K(-,K)$ and $\mathfrak{m}$ is the
kernel of the augmentation map $\epsilon$ of the Hopf algebra $K[G]$, and by 
$e_{ij}$ the elements of $Dist_1(G)$ determined by $e_{ij}(c_{hk}) =
\delta_{ih}\delta_{jk}$ and $e_{ij}(1) = 0$. Denote the parity of $e_{ij}$
to be sum of parities $|i|$ of $i$ and $|j|$ of $j$. Then $e_{ij}$ belong to
the Lie superalgebra $Lie(G)=(\mathfrak{m}/\mathfrak{m}^2)^*$ which is
identified with the general linear Lie superalgebra $\mathfrak{gl}(m|n)$.
Under this identification $e_{ij}$ corresponds to the matrix unit which has
all entries zeroes except the entry at the position $(i,j)$ which equals
one. The commutation relations for the matrix units $e_{ij}$ are given as 
\begin{equation*}
[e_{ab},e_{cd}]=e_{ad}\delta_{bc}+(-1)^{(|a|+|b|)(|c|+|d|)}
e_{cb}\delta_{ad}.
\end{equation*}

Let $U_{\mathbb{C}}$ be the universal enveloping algebra of $\mathfrak{gl}%
(m|n)$ over the field of complex numbers. Then the Kostant $\mathbb{Z}$-form 
$U_\mathbb{Z}$ is generated by elements $e_{ij}$ for odd $e_{ij}$, $%
e_{ij}^{(r)}= \frac{e_{ij}^r}{r!}$ for even $e_{ij}$, and $\binom{e_{ii}}{r}=%
\frac{e_{ii}(e_{ii}-1)\ldots (e_{ii}-r+1)}{r!}$ for all $r>0$. In what
follows we will use PBW theorem, and corresponding to our choice of the
Borel subsupergroup $B$, we order these generators as follows: odd $e_{ij}$
for $i<j$ first, followed by even $e_{ij}^{(r)}$ for $i<j$, then $\binom{%
e_{ii}}{r}$ followed by odd $e_{ij}$ for $i>j$, and then finally folowed by $%
e_{ij}^{(r)}$ for $i>j$.

Returning back to the above description of $H^0_G(\lambda)$, as the image of 
$H^0_{G_{ev}}(\lambda)\otimes S(C_{12})$, it is important to mention that it
follows from \cite{z} that the image of $\phi$ is a subset of $K[G]$, which
has a natural structure of a $G$-module, given by the superalgebra of
distributions $Dist(G)$, and this $G$-module structure coincide with the $G$%
-module structure on $H^0_G(\lambda)$.

From now on, we will view $H^0_{G_{ev}}(\lambda)$ embedded inside $%
H^0_G(\lambda)$ via $\phi$. Under this embedding, the highest vector $v_{ev}$
of $H^0_{G_{ev}}(\lambda)$ is mapped to the highest vector $%
v=\phi(v_{ev})=v^+v^-$ of $H^0_G(\lambda)$ with respect to the dominance order corresponding to $B$, 
where $v^+=v_{ev}^+$ and $%
v^-=\phi(v_{ev}^-)$. To be more precise, we describe the images under $\phi$
of the previously defined bideterminants, used in the construction of a
basis of $H^0_{ev}(\lambda)$.

If $1\leq i_1, \ldots, i_s \leq m+n$, then denote by $D^+(i_1, \ldots, i_s)$
the determinant 
\begin{equation*}
\begin{array}{|ccc|}
c_{1,i_1} & \ldots & c_{1,i_s} \\ 
c_{2,i_1} & \ldots & c_{2,i_s} \\ 
\ldots & \ldots & \ldots \\ 
c_{s,i_1} & \ldots & c_{s,i_s}%
\end{array}%
.
\end{equation*}
Clearly, if some of the numbers $i_1, \ldots, i_s$ coincide, then $%
D^+(i_1,\ldots, i_s)=0$.

Assume that $\lambda^+$ is a polynomial weight and a tableau $T(I)$ of shape 
$\lambda^+$ is such that its entry in the $a$-th row and $b$-column is $%
i_{ab}$. The bideterminant $B^+(I)$ is a product of determinants $%
D^+(i_{1b},\ldots, i_{m,b})$ for $b=1, \ldots \lambda^+_m$, $%
D^+(i_{1b},\ldots, i_{m-1,b})$ for $b=\lambda^+_m+1, \ldots \lambda^+_{m-1}$%
, \ldots $D^+(i_{1b})$ for $b=\lambda^+_2+1,\ldots, \lambda^+_1$. If we
denote the length of the $b$-th column of $T(I)$ by $\ell(b)$, we can write 
\begin{equation*}
B^+(I)=\prod_{b=1}^{\lambda^+_1}D^+(i_{1b},\ldots, i_{\ell(b),b}).
\end{equation*}

If $m+1\leq j_1, \ldots, j_s \leq m+n$, then denote by $D^-(j_1, \ldots,
j_s) $ the determinant 
\begin{equation*}
\begin{array}{|ccc|}
\phi(c_{m+1,j_1}) & \ldots & \phi(c_{m+1,j_s}) \\ 
\phi(c_{m+2,j_1}) & \ldots & \phi(c_{m+2,j_s}) \\ 
\ldots & \ldots & \ldots \\ 
\phi(c_{m+s,j_1}) & \ldots & \phi(c_{m+s,j_s})%
\end{array}%
.
\end{equation*}
Clearly, if some of the numbers $j_1, \ldots, j_s$ coincide, then $%
D^-(j_1,\ldots, j_s)=0$.

Assume that a tableau $T(J)$ of shape $\lambda^-$ is such that its entry in
the $a$-th row and $b$-column is $j_{ab}$. The image of the bideterminant $%
B^-(J)$ under $\phi$ is a product of determinants $D^-(j_{1b},\ldots,
j_{n,b})$ for $b=1, \ldots \lambda^-_n$, $D^-(j_{1b},\ldots, j_{n-1,b})$ for 
$b=\lambda^-_n+1, \ldots \lambda^-_{n-1}$, \ldots $D^-(j_{1b})$ for $%
b=\lambda^-_2+1,\ldots, \lambda^-_1$. If we denote the length of the $b$-th
column of $T(J)$ by $\ell(b)$, we can write 
\begin{equation*}
\phi(B^-(J))=\prod_{b=1}^{\lambda^-_1}D^-(j_{1b},\ldots, j_{\ell(b),b}).
\end{equation*}

The image of $H^0_{G_{ev}}(\lambda)$ inside $H^0_G(\lambda)$ has a basis
consting of product of bideterminants $D^+(i_{1b},\ldots, i_{\ell(b),b})$
and $D^-(j_{1b},\ldots, j_{\ell(b),b})$, where $D^+(1,\ldots, m)$ appears
with exponent $\lambda_m$ (possibly negative if $\lambda$ is not polynomial).

In particular, 
\begin{equation*}
v^+=\prod_{a=1}^m D^+(1,\ldots, a)^{\lambda^+_a-\lambda^+_{a+1}}, \qquad
v^-=\prod_{b=1}^n D^-(m+1, \ldots, m+b)^{\lambda^-_b - \lambda^-_{b+1}},
\end{equation*}
where $\lambda^+_{m+1}=0=\lambda^-_{n+1}$.

The (left) action of $Dist(G)$ on $A(m|n)$ can be expressed in terms of
right superderivations of $A(m|n)$.

A right superderivation $_{ij}D$ of $A(m|n)$ of parity $|i|+|j|\pmod 2$
satisfies 
\begin{equation*}
(ab)_{ij}D=(-1)^{(|i|+|j|)|b|} (a)_{ij}Db+a(b)_{ij}D
\end{equation*}
for $a,b\in A(m|n)$ and is given by $(c_{kl})_{lj}D=c_{kj}$ and $%
(c_{kl})_{ij}D=0$ for $l\neq i$.

The element $e_{kl}$ from $Lie(G)\subseteq Dist(G)$ acts on $c_{ij}$ as $%
e_{kl}\cdot c_{ij}=\delta_{jl}c_{ik}$. Since the last expression $%
\delta_{jl}c_{ik}=(c_{ij}) _{lk}D$, the action of $e_{kl}$ on $A(m|n)$ is
identical to the action of $_{lk}D$ (not $_{kl}D$) on $A(m|n)$. For more
details please consult \cite{zs}.

Moreover, both actions extends uniquely to $K[G]$ using the quotient rule
and we will identify the action of $e_{kl}$ on $K[G]$ with the action of the
superderivation $_{lk}D$ extended to $K[G]$.

Now we have an effective way how to represent elements of $H^0_G(\lambda)$
and the $G$-action on it. We will derive some basic computational formulas
in section \ref{dva}

Finally, let $P$ be the parabolic subsupergroup of $G$ consisting of matrices $C$
that have their lower $n\times m$ block $C_{21}$ equal to zero. 
The universal highest weight module (the Weyl module) $V_G(\lambda)$
can be given as $V_G(\lambda)= Dist(G)\otimes_{Dist(P)}V_{G_{ev}}(\lambda)$,
where $V_{G_{ev}}(\lambda)$ is the Weyl module for $G_{ev}$ that is viewed as 
$P$-supermodule via the epimorhism $P\to G_{ev}$.
By Proposition 5.8 of \cite{z}, $V_G(\lambda)\cong
H^0_G(\lambda)^{\langle \tau \rangle}$, where the supertransposition $\tau$
provides a contravariant duality. For the Lie superalgebra $\mathfrak{g}=%
\mathfrak{gl}(m|n)$ and its triangular decomposition $\mathfrak{g}=\mathfrak{%
g}_-\oplus \mathfrak{g}_0\oplus \mathfrak{g}^+$ define the corresponding
module as $V_{\mathfrak{g}}(\lambda)=U(\mathfrak{g})\otimes_{U(\mathfrak{g_0}%
\oplus \mathfrak{g}_+)} V_{\mathfrak{g}_0}(\lambda)$, where $V_{\mathfrak{g}%
_0}(\lambda)$ is the $\mathfrak{g}_0$-Weyl module of the highest weight $%
\lambda$ on which $\mathfrak{g}_+$ acts trivially and $U$'s denote universal
enveloping algebras. In the case when the characteristic of $K$ equals zero,
this coincides with the Kac module $K(\lambda)= U(\mathfrak{g})\otimes_{U(%
\mathfrak{g}_0\oplus \mathfrak{g}_+)} L_{\mathfrak{g}_0}(\lambda)$, where $%
L_{\mathfrak{g}_0}(\lambda)$ is the simple $\mathfrak{g}_0$-module of the
highest weight $\lambda$ on which $\mathfrak{g}_+$ acts trivially.

\section{Basic formulas}

\label{dva}

Denote by $A=\left(%
\begin{array}{cccc}
A_{11} & A_{12} & \ldots & A_{1m} \\ 
A_{21} & A_{22} & \ldots & A_{2m} \\ 
\ldots & \ldots & \ldots & \ldots \\ 
A_{m1} & A_{m2} & \ldots & A_{mm}%
\end{array}%
\right)$ the adjoint of the matrix $C_{11}$. Then $C_{11}^{-1}=\frac{1}{D}A$
and 
\begin{equation*}
y_{ij}=\phi(c_{ij})=\frac{A_{i1}c_{1j}+A_{i2}c_{2j}+\ldots +A_{im}c_{mj}}{D}
\end{equation*}
for $1\leq i\leq m$ and $m+1\leq j \leq m+n$. Moreover, for $m+1\leq k,l
\leq m+n$ we have 
\begin{equation*}
\phi(c_{kl})=c_{kl}-c_{k1}y_{1l} - \ldots -c_{km}y_{ml}.
\end{equation*}

We will determine action of superderivations $_{kl}D$, where $1\leq k\leq m$
and $m+1\leq l\leq m+n$, on various elements of $K[G]$.

\begin{lm}
\label{1} If $1\leq i,k \leq m$ and $m+1\leq j,l \leq m+n$, then 
\begin{equation*}
(y_{ij})_{kl}D=y_{il}y_{kj}.
\end{equation*}
\end{lm}

\begin{proof}
First we show that $(D)_{kl}D=Dy_{kl}$. Write $D=A_{k1}c_{1k}+\ldots
+A_{km}c_{mk}$. Since $(-1)^{k+j}A_{kj}$ is the determinant of the matrix
obtained by removing the $j$-th row and $k$th column from $C_{11}$, it
follows that $A_{kj}$ is a sum of monomials in the variables $c_{rs}$ for $%
r\neq j$ and $s\neq k$. Therefore by the superderivation property of $_{kl}D$
we infer that $(A_{kj})_{kl}D=0$. Since $(c_{ik})_{kl}D=c_{il}$ we conclude
that $(D)_{kl}D=A_{k1}c_{1l}+\ldots +A_{km}c_{ml}=Dy_{kl}$.

Assume that $i=k$. Since $Dy_{kj}=A_{k1}c_{1j}+\ldots +A_{km}c_{mj}$, $%
(A_{ka})_{kl}D=0$ and $(c_{aj})_{kl}D=0$ for every $a=1, \ldots m$, using
the superderivation property of $_{kl}D$ we derive $(Dy_{kj})_{kl}D=0$. Then 
$(Dy_{kj})_{kl}D=-Dy_{kl}y_{kj}+D(y_{kj})_{kl}D$ implies that $%
(y_{kj})_{kl}D=y_{kl}y_{kj}$.

Therefore we can assume $m\geq 2$. For $a\neq c$ and $b\neq d$, denote by $%
M(ab|cd)$ the $(m-2)\times (m-2)$ minor of the matrix $C_{11}$ obtained by
deleting $a$-th and $b$-th row and $c$-th and $d$-th columns. If $a=b$ or $%
c=d$, then set $M(ab,cd)=0$.

Now assume that $i<k$. Expanding the determinant representing $A_{ib}$ by
the column containing entries from the $k$-th column of the matrix $C_{11}$
we obtain 
\begin{equation*}
(Dy_{ij})_{kl}D=(A_{i1}c_{1j}+\ldots + A_{im}c_{mj})_{kl}D
\end{equation*}
\begin{equation*}
=\sum_{b=1}^m \sum_{a=1}^{b-1} (-1)^{a+b+k+i} M(ab, ki) c_{al}c_{bj}
-\sum_{b=1}^m \sum_{a=b+1}^{m} (-1)^{a+b+k+i} M(ab, ki) c_{al}c_{bj}.
\end{equation*}

On the other hand,

\begin{equation*}
(Dy_{il})(Dy_{kj})=(A_{i1}c_{1l}+\ldots + A_{im}c_{ml})(A_{k1}c_{1j}+\ldots
+ A_{km}c_{mj})= \sum_{b=1}^m \sum_{a=1}^m A_{ia}A_{kb}c_{al}c_{bj}
\end{equation*}
and 
\begin{equation*}
(Dy_{kl})(Dy_{ij})=(A_{k1}c_{1l}+\ldots + A_{km}c_{ml})(A_{i1}c_{1j}+\ldots
+ A_{im}c_{mj})= \sum_{b=1}^m \sum_{a=1}^m A_{ka}A_{ib}c_{al}c_{bj}
\end{equation*}
implies that 
\begin{equation*}
(Dy_{il})(Dy_{kj})-(Dy_{kl})(Dy_{ij})=\sum_{b=1}^m \sum_{a=1}^m
(A_{ia}A_{kb}-A_{ka}A_{ib})c_{al}c_{bj}.
\end{equation*}

Using the Jacobi Theorem on minors of the adjoint matrix (see \cite{g},
p.21, \cite{f}, p.57 or Theorem 2.5.2 of \cite{p}) we have that $%
A_{ia}A_{kb}-A_{ka}A_{ib}=(-1)^{a+b+k+i} D M(ab, ki)$ for $a<b$ and $%
A_{ia}A_{kb}-A_{ka}A_{ib}=(-1)^{a+b+k+i+1} D M(ab, ki)$ for $a>b$.

Therefore 
\begin{equation*}
(Dy_{il})(Dy_{kj})-(Dy_{kl})(Dy_{ij})=D(Dy_{ij})_{kl}D=-D^2y_{kl}y_{ij}+D^2(y_{ij})_{kl}D
\end{equation*}
and this implies $(y_{ij})_{kl}D=y_{il}y_{kj}$.

The remaining case $i>k$ can be handled analogously.
\end{proof}

\begin{lm}
\label{2} If $1\leq k \leq m$ and $m+1\leq i,j,l \leq m+n$, then 
\begin{equation*}
(\phi(c_{ij}))_{kl}D=\phi(c_{il})y_{kj}.
\end{equation*}
\end{lm}

\begin{proof}
Since $\phi(c_{ij})=c_{ij}-c_{i1}y_{1j} - \ldots -c_{im}y_{mj}$, using Lemma %
\ref{1} we compute 
\begin{equation*}
(\phi(c_{ij}))_{kl}D=-c_{i1}y_{1,l}y_{kj} - c_{i2}y_{2l}y_{kj} -\ldots
+c_{il}y_{kj} - c_{ik}y_{kl}y_{kj} -\ldots -c_{im}y_{ml}y_{kj}=
\phi(c_{il})y_{kj}.
\end{equation*}
\end{proof}

\begin{lm}
\label{3} If $1\leq k \leq m$ and $m+1\leq l,j_1, \ldots, j_s \leq m+n$,
then 
\begin{equation*}
(D^-(j_1, \ldots, j_s))_{kl}D=D^-(l,j_2, \ldots, j_s) y_{kj_1}+D^-(j_1,l,
\ldots, j_s)y_{kj_2}+ \ldots + D^-(j_1, \ldots, j_{s-1},l)y_{kj_s}.
\end{equation*}
\end{lm}

\begin{proof}
Using Lemma \ref{2}, the proof is straightforward.
\end{proof}

Denote by $_{al}\tilde{D}$ the operator that acts on determinant $D^-(j_1,
\ldots, j_s)$ with distinct entries $j_1, \ldots, j_s$ by replacing the
single entry which is equal to $a$ by $l$, say $(D^-(j_1, j_2, \ldots, j_s)) %
\phantom{ }_{j_1,l}D= D^-(l,j_2, \ldots, j_s)$ and so on. Then the above
formula can be written as 
\begin{equation*}
\begin{array}{ll}
(D^-(j_1, \ldots, j_s))_{kl}D= & D^-(j_1,j_2, \ldots, j_s)_{j_1l}\tilde{D}
y_{kj_1}+D^-(j_1,j_2, \ldots, j_s)_{j_2l}\tilde{D}y_{kj_2}+ \\ 
& \ldots + D^-(j_1, \ldots, j_{s-1},j_s)_{j_sl}\tilde{D}y_{kj_s}%
\end{array}%
\end{equation*}
or 
\begin{equation*}
(D^-(j_1, \ldots, j_s))_{kl}D=D^-(j_1,j_2, \ldots, j_s)(\sum_{a=m+1}^{m+n} %
\phantom{ }_{al}\tilde{D} y_{ka})
\end{equation*}
when we define $(D^-(j_1, j_2, \ldots, j_s)) \phantom{ }_{a,l}\tilde{D}=0$
if $a$ does not equal any one of the entries $j_1, \ldots, j_s$.

Extend the definition of the operators $_{al}\tilde{D}$ to products and make
it a derivation.

\begin{lm}
\label{4} Let $1\leq k \leq m$, $m+1\leq l \leq m+n$ and $T(J)$ be a
tableaux of shape $\lambda^-$ such that all of its entries are bigger than $%
m+1$. Then 
\begin{equation*}
(B^-(J))\phantom{ }_{kl}D=\sum_{a=m+1}^{m+n} (B^-(J)) \phantom{ }_{al}\tilde{%
D} y_{ka}.
\end{equation*}
\end{lm}

\begin{proof}
The proof follows immediately from properties of $_{kl}D$.
\end{proof}

\section{Image of $v$ under $_{m,m+1}D \ldots _{m,m+n}D$}

\begin{lm}
\label{5} Let $m\leq l \leq m+n$. Then $(v^+)_{ml}D=\lambda^+_{m}v^+y_{ml}$.
\end{lm}

\begin{proof}
Assume first that $\lambda^+$ is a polynomial weight. Then the bideterminant 
$B^+(T_{\lambda}^+)=\prod_{b=1}^{\lambda^+_1}D^+(1,\ldots,\ell(b))$. The
image 
\begin{equation*}
D^+(1,\dots, m)_{ml}D=(A_{m1}c_{1m}+ A_{m2}c_{2m} + \ldots
A_{mm}c_{mm})_{ml}D
\end{equation*}
\begin{equation*}
=A_{m1}c_{1l}+A_{m2}c_{2l}+ \ldots + A_{mm}c_{ml}=D^+(1,\ldots, m) y_{ml}
\end{equation*}
and consequently, 
\begin{equation*}
(D^+(1,\dots, m)^{\lambda^+_m})_{ml}D=\lambda^+_m (D^+(1,\dots,
m))^{\lambda^+_m} y_{ml}.
\end{equation*}
Since $D^+(1,\dots, s)_{ml}D=0$ for $s<m$, the claim follows.

If $\lambda^+$ is not a polynomial weight, then $(v^{++})_{ml}D=%
\lambda^{++}_m v^{++}y_{ml}=0$ and $(v^+)_{ml}D=(v^{++}D^{%
\lambda^+_m})_{ml}D=v^{++}(D^{\lambda^+_m})_{ml}D=v^{++}\lambda^+_mD^{%
\lambda^+_m}y_{ml}= \lambda^+_mv^+y_{ml}$.
\end{proof}

\begin{lm}
\label{6} Let $1\leq k\leq m$ and $m\leq l \leq m+n$. If $l\leq m+s$, then $%
(D^-(m+1, \ldots, m+s))_{k,l}D=D^-(m+1, \ldots, m+s)y_{kl}$. If $m+s<l$,
then $(D^-(m+1, \ldots, m+s))_{k,l}D$ is a linear combination of $y_{k,m+1}$
through $y_{k, m+s}$.
\end{lm}

\begin{proof}
The fact that $(D^-(m+1, \ldots, m+s))_{k,l}D$ is a linear combination of $%
y_{k,m+1}$ through $y_{k, m+s}$ follows from Lemma \ref{3}. If $l\leq m+s$,
then all but one summand in the formula in Lemma \ref{3} vanishes because of
repetition of entries in $D^-$. The only remaining summand is $D^-(m+1,
\ldots, m+s)y_{kl}$.
\end{proof}

For each $j=1, \ldots, m$ denote by $_{j}D$ the composition of
superderivations

\noindent $_{j,m+1}D\ldots _{j,m+s}D \ldots _{j, m+n}D$, by $y_j$ the
product $y_{j,m+1}\ldots y_{j,m+n}$, and by $\omega_j(\lambda)$ the product $%
\prod_{i=1}^n \omega_{j,i}(\lambda))$.

\begin{pr}
\label{7} 
\begin{equation*}
(v)_{m}D=\omega_{m}(\lambda) vy_{m}.
\end{equation*}
\end{pr}

\begin{proof}
We will proceed by induction. Since 
\begin{equation*}
B^-(T_{\lambda}^-)=\prod_{b=1}^{\lambda^-_1}D^-(m+1,\ldots, m+\ell(b))
\end{equation*}
and each 
\begin{equation*}
(D^-(m+1,\ldots, m+\ell(b)))_{m,m+1}D=(D^-(m+1,\ldots, m+\ell(b))y_{m,m+1}
\end{equation*}
by Lemma \ref{6}, we infer that $(v^-)_{m,m+1}D=\lambda^-_1 v^-$. Using
Lemma \ref{5} we conclude that $(v)_{m,m+1}D=(\lambda^+_m+\lambda^-_{1})v
y_{m,m+1}$.

Next, assume that we have already proved the following formula when we apply 
$s$ superderivations 
\begin{equation*}
(v)_{m,m+1}D\ldots _{m,m+s}D=(\lambda^+_m+\lambda^-_{1}) \ldots
(\lambda^+_m+\lambda^-_{s}-s+1)vy_{m,m+1}\ldots y_{m,m+s}
\end{equation*}
and consider the expression $E=(vy_{m,m+1}\ldots y_{m,m+s})_{m,m+s+1}D$. By
Lemma \ref{5} we have $(v^+)_{m,m+s+1}D=\lambda^+_mv^+y_{m,m+s+1}$ and by
Lemma \ref{6} we obtain that $(v^-)_{m,m+s+1}D-\lambda^-_{s+1}v^-y_{m,m+s+1}$
is a linear combination of elements $y_{m,m+1}$ through $y_{m+s}$. Therefore 
\begin{equation*}
(v)_{m,m+s+1}Dy_{m,m+1}\ldots
y_{m,m+s}=(\lambda^+_m+\lambda^-_{s+1})vy_{m,m+s+1}y_{m,m+1}\ldots y_{m,m+s}.
\end{equation*}
Using Lemma \ref{1} we infer 
\begin{equation*}
(y_{m,m+1}\ldots
y_{m,m+s})_{m,m+s+1}D=(-1)^{s-1}(y_{m,m+1})_{m,m+s+1}Dy_{m,m+2}\ldots
y_{m,m+s} +
\end{equation*}
\begin{equation*}
(-1)^{s-2}y_{m,m+1}(y_{m,m+2})_{m,m+s+1}Dy_{m,m+3}\ldots y_{m,m+s}+ \ldots +
y_{m,m+1}\ldots (y_{m,m+s+1})_{m,m+s+1}D=
\end{equation*}
\begin{equation*}
(-1)^{s-1}y_{m,m+s}y_{m,m+1}y_{m,m+2}\ldots y_{m,m+s} +
(-1)^{s-2}y_{m,m+1}y_{m,m+s}y_{m,m+2}y_{m,m+3}\ldots y_{m,m+s}
\end{equation*}
\begin{equation*}
+\ldots + y_{m,m+1} \ldots y_{m,m+s-1}y_{m,m+s+1}y_{m,m+s}=-sy_{m,m+1}\ldots
y_{m,m+s}y_{m,m+s+1}.
\end{equation*}

Finally, 
\begin{equation*}
E=(-1)^s(v)_{m,m+s+1}Dy_{m,m+1}\ldots y_{m,m+s}+v(y_{m,m+1}\ldots
y_{m,m+s})_{m,m+s+1}D
\end{equation*}
\begin{equation*}
=(\lambda^+_m+\lambda^-_{s+1}-s)vy_{m,m+1}\ldots y_{m,m+s+1}
\end{equation*}
which shows that our formula is valid when we apply $s+1$ superderivations.
The induction concludes the proof.
\end{proof}

\section{Polynomial weight $\protect\lambda$ in the case $char K=0$}

Since we are able to fully address the case of a polynomial weight $\lambda$
when the characteristic $K$ equals zero, we interrupt our progress toward
the proof of Theorem 1 in order to make few observations. In this section we
assume that the characteristic of $K$ is zero and $\lambda$ is a polynomial
weight of $G$.

A partition $\lambda=(\lambda_1 \geq \ldots \geq \lambda_k)$ is called a $%
(m|n)$-hook partition if $\lambda_{m+1}\leq n$. The concept of hook Schur
functions $HS_{\lambda}(x_1,\dots,x_m;y_1,\dots, y_n)$ for $(m|n)$-hook
partitions $\lambda$ , as a generalization of classical Schur functions, was
defined by Berele and Regev in \cite{br} in connection to the representation
theory of general linear Lie superalgebras over a field of characteristic
zero. They have proved the basic properties of hook Schur functions,
including the following factorization formula, see Theorem 6.20 of \cite{br}.

Assume the partition $\lambda$ satisfies $\lambda_m\geq n$. Define the
partitions $\mu=(\mu_1, \ldots, \mu_m)$ and $\nu^{\prime}=(\nu^{\prime}_1,
\ldots, \nu^{\prime}_{k-m})$ in such a way that $\lambda_i=\mu_i+n$ for $%
i=1, \ldots, m$ and $\nu^{\prime}_j=\lambda_{m+j}$ for $j=1, \ldots, k-m$,
and denote by $\nu$ the conjugate of the partition $\nu^{\prime}$. Then 
\begin{equation*}
HS_{\lambda}(x_1,\dots,x_m;y_1,\dots, y_n)=\prod_{i=1}^m\prod_{j=1}^n
(x_i+y_j) S_{\mu}(x_1, \ldots, x_m) S_{\nu}(y_1, \ldots, y_n),
\end{equation*}
where $S_{\mu}$ and $S_{\nu}$ are the Schur functions corresponding to the
partitions $\mu$ and $\nu$. Different proofs of this formula were given in 
\cite{r}, \cite{pt} and \cite{gg}. This formula was further generalized to
the Sergeev-Pragacz formula, see Exercise 24 in I.3 of \cite{md}, \cite{pt}, 
\cite{mj} or \cite{b}.

Let $(\lambda^+|\lambda^-)$ be the highest weight of a simple $G$-module $%
L(\lambda^+|\lambda^-)$. In the case $char K=0$ the highest weights of
simple $G$-modules correspond to $(m|n)$-hook partitions in the following
way. If $\lambda$ is a $(m|n)$-hook partition, then $\lambda^+$ consists of
the first $m$ components of $\lambda$ and $\lambda^-$ is the conjugate
partition to the remaining components of $\lambda$. If the characteristic of 
$K$ is zero, then we will identify the $(m|n)$-hook weight $\lambda$ with
the corresponding $GL(m|n)$-weight $(\lambda^+|\lambda^-)$.

The hook Schur function $HS_{\lambda }(x_{1},\dots ,x_{m};y_{1},\dots
,y_{n}) $ coincides with the character of $L_{G}(\lambda )$. According to
Corollary 5.4 of \cite{z}, the character of the induced module $%
H_{G}^{0}(\lambda )$ is given as 
\begin{equation*}
\prod_{i=1}^{m}\prod_{j=1}^{n}(1+\frac{y_{j}}{x_{i}})S_{\lambda
^{+}}(x_{1},\ldots ,x_{m})S_{\lambda ^{-}}(y_{1},\ldots ,y_{n}).
\end{equation*}%
If $\lambda $ satisfies $\lambda _{m}\geq n$, then $S_{\lambda
^{+}}(x_{1},\ldots ,x_{m})=(x_{1}\ldots x_{m})^{n}S_{\mu }(x_{1},\ldots
,x_{m})$. Since $S_{\nu }(y_{1},\ldots ,y_{n})=S_{\lambda ^{-}}(y_{1},\ldots
,y_{n})$, the factorization formula of Berele and Regev implies $%
L_{G}(\lambda )=H_{G}^{0}(\lambda )$. Therefore this formula gives a
sufficient condition for the irreducibility of $H_{G}^{0}(\lambda )$.

We now prove that the condition $\lambda_m\geq n$ is also necessary for $%
H^o_G(\lambda)$ to be irreducible.

%To determine when is $H^0_G(\la)$ irreducible, we need to dinf a necessary and sufficient condition for the equality
%\[HS_{\la}(x_1,\dots,x_m;y_1,\dots, y_n)=\prod_{i=1}^m\prod_{j=1}^n (1+\frac{y_j}{x_i}) S_{\la^+}(x_1, \ldots, x_m) S_{\la^-}(y_1, \ldots, y_n).\]

\begin{teo}
\label{8} Let $char K=0$, $\lambda$ be a $(m|n)$-hook partition, and $%
H^0_G(\lambda)$ be the induced module. Then $H^0_G(\lambda)$ is irreducible
if and only if $\lambda_m\geq n$.
\end{teo}

\begin{proof}
If $\lambda_m \geq n$, then Theorem 6.20 of \cite{br} implies that the
characters of $L(\lambda)$ and $H^0_G(\lambda)$ coincide, hence $%
H^0(\lambda)\cong L(\lambda)$.

Assume now that $\lambda_m<n$. Then $m<m+\lambda_m+1\leq m+n$, $%
\lambda_{m+\lambda_m+1}=0$ and $\lambda_m+ \lambda_{m+\lambda_m+1}
-(\lambda_m+1)+1 =0$. By Proposition \ref{7}, $(v)_{m}D=0$. On the other
hand, the previous character formula for $H^0_G(\lambda)$ implies that the
weight space of $H^0_G(\lambda)$ corresponding to the weight $%
\kappa=\lambda+(0, \ldots, 0, -n|1,\ldots, 1)$ is one-dimensional. Using
commutation relations for the matrix units $e_{ij}$ we infer that odd
superderivations $_{mj}D$ for $m<j$ supercommute. Therefore, applying PBW
theorem with the ordering of the generators of $Dist(G)$ as in Section 1  and 
using the fact the $v$ is the highest vector of $H^0_G(\lambda)$, we
conclude that the weight space $H^0_G(\lambda)_{\kappa}$ is generated by $(v)_mD$. Thus $%
(v)_{m}D=0$ implies that $L_G(\lambda)$ is a proper submodule of $H^0_G(\lambda)$%
.
\end{proof}

\section{Image of $v$ under $_{m}D \ldots _{1}D$}

In order to proceed we need a few definitions.

Let $M=(m_{ij})$ be a matrix of size $s\times s$. Let $\{i_{1},\ldots
,i_{t}\}$ and $\{j_{1},\ldots ,j_{t}\}$ be sequences of elements from $%
\{1,\ldots ,s\}$. Denote by $M(i_{1},\ldots ,i_{t}|j_{1},\ldots j_{t})$ the
determinant of the matrix of size $t\times t$ such that its entry in the $a$%
-th row and the $b$-th column is $m_{i_{a},j_{b}}$. If the entries in $%
\{i_{1},\ldots ,i_{t}\}$ and $\{j_{1},\ldots ,j_{t}\}$ are pair-wise
different, then $M(i_{1},\ldots ,i_{t}|j_{1},\ldots j_{t})$ is the $t$-th
minor of $M$ corresponding to the rows $\{i_{1},\ldots ,i_{t}\}$ and columns 
$\{j_{1},\ldots ,j_{t}\}$ of matrix $M$. We will use this notation for the
matrix $C$ and further denote $C(1,2,\ldots ,t|j_{1},\ldots ,j_{t})$ by $%
C(j_{1},\ldots ,j_{t})$.

\begin{lm}
\label{9'} Let $1\leq i\leq j\leq m$ and $m+1\leq l\leq m+n$. Then 
\begin{equation*}
C(1, \ldots, i-1, l, i+1, \ldots, j)
\end{equation*}
\begin{equation*}
=C(1, \ldots, j)y_{i,l} +(-1)^{i+j} \left[C(1,\ldots, \hat{i}, j,
j+1)y_{j+1,l} + \ldots + C(1, \ldots, \hat{i}, j, m) y_{m,l}\right].
\end{equation*}
\end{lm}

\begin{proof}
Use the Laplace expansion for $C(1, \ldots, i-1, l, i+1, \ldots, j)$ along
its $i$-th column to get 
\begin{equation*}
C(1, \ldots, i-1, l, i+1, \ldots, j)= \sum_{s=1}^{j} (-1)^{i+s} C(1, \ldots, 
\hat{s}, \ldots, j|1, \ldots, \hat{i}, \ldots j) c_{sl}.
\end{equation*}
Expand the right-hand side to 
\begin{equation*}
\sum_{s=1}^m \left[C(1,\ldots, j) A_{is} +(-1)^{i+j} C(1,\ldots, \hat{i}, j,
j+1) A_{j+1,s} + \ldots + (-1)^{i+j}C(1, \ldots, \hat{i}, j, m) A_{ms}\right]%
\frac{c_{sl}}{D_{11}}.
\end{equation*}

Therefore, we need to show that 
\begin{equation*}
C(1,\ldots, j) A_{is} + (-1)^{i+j}C(1,\ldots, \hat{i}, j, j+1) A_{j+1,s} +
\ldots +(-1)^{i+j} C(1, \ldots, \hat{i}, j, m) A_{ms}
\end{equation*}
%\[=\sum_{k=j}^m (-1)^{k+s} C(1,\ldots, j-1, k) C(1, \ldots, \hat{s}, \ldots, m|1, \ldots, \hat{k}, \ldots m) \]
equals $(-1)^{i+s} C(1,\ldots, \hat{s}, \ldots, j|1, \ldots, \hat{i},
\ldots, j)D_{11}$ if $s\leq j$, and zero otherwise.

These equalities are obtained by using the method of extension (Muir's law
of extensible minors) of determinantal identities - see Sections 7 an 8 of 
\cite{bs}. Assume $s\leq j$. The Laplace expansion of $C(s,j+1, \ldots,
m|i,j+1, \ldots, m)$ along the first row gives 
\begin{equation*}
C(\emptyset|\emptyset)C(s,j+1, \ldots, m|i,j+1, \ldots, m)
\end{equation*}
\begin{equation*}
=C(s|i)C(j+1, \ldots, m|j+1,\ldots, m)
\end{equation*}
\begin{equation*}
+\sum_{k=j+1}^m (-1)^{k+j} C(s,k)C(j+1, \ldots, m|i,j+1, \ldots, \hat{k},
\ldots, m).
\end{equation*}
By adjoining the symbols $(1, \ldots, \hat{s}, \ldots, j|1, \ldots, \hat{i},
\ldots, j)$ we obtain 
\begin{equation*}
C(1, \ldots, \hat{s}, \ldots, j|1, \ldots,\hat{i}, \ldots j)C(1, \ldots, 
\hat{s}, \ldots, j, s,j+1, \ldots, m|1, \ldots, \hat{i}, \ldots, j, i, j+1,
\ldots, m)
\end{equation*}
\begin{equation*}
=C(1, \ldots, \hat{s}, \ldots, j,s|1, \ldots, \hat{i}, \ldots, j,i) C(1,
\ldots, \hat{s}, \ldots, m|1, \ldots, \hat{i}, \ldots, j,j+1,\ldots, m)
\end{equation*}
\begin{equation*}
+\sum_{k=j+1}^m (-1)^{k+j} C(1, \ldots, \hat{s}, \ldots, j,s|1, \ldots, \hat{%
i}, \ldots, j,k) C(1, \ldots, \hat{s}, \ldots m|1, \ldots, \hat{i}, \ldots,
j,i,j+1, \ldots, \hat{k}, \ldots, m).
\end{equation*}
Therefore 
\begin{equation*}
(-1)^{s+i}C(1, \ldots, \hat{s}, \ldots, j|1, \ldots,\hat{i}, \ldots j)D_{11}=
\end{equation*}
\begin{equation*}
=(-1)^{s+i}C(1,\ldots, j)C(1, \ldots, \hat{s},\ldots, m|1, \ldots, \hat{i},
\ldots, m)
\end{equation*}
\begin{equation*}
+(-1)^{s+i}\sum_{k=j+1}^m (-1)^{k+j} C(1, \ldots, \hat{i}, \ldots, j,k) C(1,
\ldots, \hat{s},\ldots, m|1, \ldots,\hat{k}, \ldots, m)
\end{equation*}
and 
\begin{equation*}
C(1, \ldots, \hat{s}, \ldots, j|1, \ldots,\hat{i}, \ldots j)D_{11}
\end{equation*}
\begin{equation*}
=C(1,\ldots, j)A_{is}+(-1)^{i+s}\sum_{k=j+1}^m C(1, \ldots, \hat{i}, \ldots,
j,k) A_{ks}.
\end{equation*}

If $s>j$, then we can use the Laplace expansion of $C(s,j+1, \ldots,
m|i,j+1, \ldots, m)$ along the first row. However, in this case $C(s,j+1,
\ldots, m|i,j+1, \ldots, m)=0$ and $C(1,\ldots,
j)A_{is}+(-1)^{i+s}\sum_{k=j+1}^m C(1, \ldots, \hat{i}, \ldots, j,k)
A_{ks}=0 $.
\end{proof}

\begin{lm}
\label{10} Let $1\leq i \leq m$ and $m+1\leq l\leq m+n$. Then $%
(v^+)_{il}D-\lambda^+_mv^+y_{il}$ is a sum of multiples of $y_{jl}$ for $j>i$%
.
\end{lm}

\begin{proof}
Assume first that $\lambda^+$ is a polynomial weight. If $j<i$, then $D^+(1,
\ldots, j)_{il}D=0$. If $j\geq i$, then 
\begin{equation*}
D^+(1,\ldots, j)_{il}D=C(1, \ldots, i-1, l, i+1, \ldots, j)=
\end{equation*}
\begin{equation*}
=D^+(1, \ldots, j)y_{i,l} +(-1)^{i+j} \left[C(1,\ldots, \hat{i}, j,
j+1)y_{j+1,l} + \ldots + C(1, \ldots, \hat{i}, j, m) y_{m,l}\right]
\end{equation*}
by Lemma \ref{9'}.

Since $v^+=B^+(T_{\lambda}^+)=\prod_{b=1}^{\lambda^+_1}D^+(1,\ldots,%
\ell(b))= \prod_{k=1}^m D^+(1,\ldots,k)^{\lambda^+_k-\lambda^+_{k+1}}$, we
infer that $(v^+)_{il}D-\lambda_iv^+y_{il}$ is a sum of multiples of $y_{jl}$
for $j>i$.

Next, assume $\lambda^+$ is not polynomial. Then 
\begin{equation*}
(v^+)_{il}D=(v^{++}D^{\lambda^+_m})_{il}D=(v^{++})_{il}D
D^{\lambda^+_m}+v^{++}\lambda^+_mD^{\lambda^+_m}y_{il}
\end{equation*}
equals a sum of multiples of $y_{jl}$ for $j>i$ and of 
\begin{equation*}
\lambda^{++}_iv^{++}y_{il}D^{\lambda^+_m}+
\lambda^+_mv^{++}D^{\lambda^+_m}y_{il}=
\lambda^{+}_iv^{++}D^{\lambda^+_m}y_{il}=\lambda^+_iv^+y_{il}.
\end{equation*}
\end{proof}

\begin{lm}
\label{11} For every $1\leq i \leq m$ and $1\leq s\leq n$ we have 
\begin{equation*}
(y_m\ldots y_{i+1} y_{i,m} \ldots y_{i,m+s-1}) _{i,m+s}D=(m-i-s+1)y_m\ldots
y_{i+1} y_{i,m} \ldots y_{i,m+s}.
\end{equation*}
\end{lm}

\begin{proof}
If $j>i$, then $(y_m)_{i,m+s}D=y_my_{i,m+s}$ follows from Lemma \ref{1}.
Therefore $(y_m\ldots y_{i+1}) _{i,m+s}D= (m-i) y_m \ldots y_{i+1} y_{i,m+s}$%
. As in the proof of Proposition \ref{7} we derive that $(y_{i,m} \ldots
y_{i,m+s-1}) _{i,m+s}D= -(s-1)y_{i,m} \ldots y_{i,l-1} y_{i,m+s}$. The claim
follows.
\end{proof}

\begin{pr}
\label{12} 
\begin{equation*}
(v)_{m}D \ldots _{1}D= (\prod_{i=1}^m \omega_{i}(\lambda)) v y_{m}\ldots
y_{1}.
\end{equation*}
\end{pr}

\begin{proof}
We proceed by induction and show a more general formula 
\begin{equation*}
(v)_{m}D \ldots _{i+1}D_{i,m+1}D \ldots _{i,m+s}D= \prod_{k=i+1}^m
\omega_{k}(\lambda) \prod_{j=1}^{s} \omega_{i,j}(\lambda) v y_m\ldots
y_{i+1}y_{i,m+1}\ldots y_{i,m+s}
\end{equation*}
for $1\leq i \leq m$ and $1\leq s\leq n$.

This formula was already verified for $i=m$ and arbitrary $s$ in the proof
of Proposition \ref{7}. First we will assume that the formula is valid for
some $i$ and $s<n$ and prove it for the indices $i$ and $s+1$. By Lemma \ref%
{10} $(v^+)_{i,m+s+1}D =\lambda^+_{i}v^+y_{i,m+s+1}$ modulo a linear
combination of $y_{j,m+s+1}$ for $j>i$, and by Lemma \ref{6} $%
(v^-)_{i,m+s+1}D=\lambda^-_{s+1}v^-y_{i,m+s+1}$ modulo a linear combination
of $y_{i,m+1}, \ldots y_{i, m+s}$. Using these and Lemma \ref{11} we compute 
\begin{equation*}
(v y_m\ldots y_{i+1}y_{i,m+1}\ldots y_{i,m+s})_{i,m+s+1}D
\end{equation*}
\begin{equation*}
= (\lambda^+_i +\lambda^-_{s+1} +m-i-s)v y_m\ldots y_{i+1}y_{i,m+1}\ldots
y_{i,m+s+1}
\end{equation*}
\begin{equation*}
=\omega_{i,s+1}(\lambda)v y_m\ldots y_{i+1}y_{i,m+1}\ldots y_{i,m+s+1}
\end{equation*}
and the formula follows by induction.

Next we will assume that the formula is valid for certain $i$ and $s=n$ and
prove it for the indices $i+1$ and $s=1$. All the arguments remain analogous
and we determine 
\begin{equation*}
(v y_m\ldots y_{i+1})_{i,m+1}D = (\lambda^+_i +\lambda^-_{1} +m-i)v
y_m\ldots y_{i+1}y_{i,m+1} =\omega_{i,1}(\lambda)v y_m\ldots y_{i+1}y_{i,m+1}
\end{equation*}
which implies the formula in this case as well.
\end{proof}

\section{The conclusion}

Proposition 5.6. of \cite{z} shows that $H^0_G(\lambda)\neq 0$ if and only if $\lambda$ is dominant. 

We require the following lemma that describes action of even
superderivations $_{kl}D$ on elements $y_{ij}$.

\begin{lm}
\label{13} Let $1\leq i \leq m$ and $m+1\leq j\leq m+n$. If $m+1\leq k,l
\leq m+n$, then $(y_{ij})_{kl}D=\delta_{jk}y_{il}$. If $1\leq k,l \leq n$,
then $(y_{ij})_{il}D=0$ and $(y_{ij})_{kl}D=-\delta_{il}y_{kj}$ for $k\neq i$%
. Consequently, $(y_m\ldots y_1)_{kl}D^{(r)}=0$ for each $r>0$ and even
superderivation $_{kl}D$ such that $k\neq l$.
\end{lm}

\begin{proof}
First observe that in each case $(D)_{kl}D=0$.

Assume now $m+1\leq k,l \leq m+n$. If $k\neq j$, then $%
(Dy_{ij})_{kl}D=(A_{i1}c_{1j}+ \ldots +A_{im}c_{mj})_{kl}D=0$ which implies $%
(y_{ij})_{kl}D=0$. If $k=j$, then $%
D(y_{ij})_{jl}D=(Dy_{ij})_{jl}D=(A_{i1}c_{1j}+ \ldots
+A_{im}c_{mj})_{jl}D=(A_{i1}c_{1l}+ \ldots +A_{im}c_{mj})_{kl}=Dy_{il}$
implies $(y_{ij})_{jl}D=y_{il}$.

Next, assume $1\leq k,l \leq n$. Applying $_{kl}D$ to cofactors $A_{ij}$ we
obtain $(A_{ij})_{il}D=A_{ij}$ and $(A_{ij})_{kl}D=-\delta_{il} A_{kj}$ for $%
k\neq i$. Then $(Dy_{ij})_{il}D=(A_{i1}c_{1j}+ \ldots +A_{im}c_{mj})_{il}D=0$
and $(Dy_{ij})_{kl}D=(A_{i1}c_{1j}+ \ldots
+A_{im}c_{mj})_{kl}D=-\delta_{il}(A_{k1}c_{1j}+ \ldots
+A_{km}c_{mj})=-\delta_{il}Dy_{kj}$ which imply the claim.
\end{proof}

Now we are ready to prove our main result.

\medskip

\textbf{Theorem 1.} \textit{The $G$- module $H^0_G(\lambda)$ is irreducible
if and only if $H^0_{G_{ev}}(\lambda)$ is an irreducible $G_{ev}$-module and 
$\lambda$ is typical.}

\begin{proof}
Assume that $H^0_G(\lambda)$ is irreducible. It is clear that $%
H^0_{G_{ev}}(\lambda)$ is an irreducible $G_{ev}$-module. If $\lambda$ is
atypical, then Proposition \ref{12} shows that $(v)_{m}D \ldots _{1}D=0$. On
the other hand, the character formula for $H^0_G(\lambda)$ implies that the
weight space of $H^0_G(\lambda)$ corresponding to the weight $%
\kappa=\lambda+(-n, \ldots,-n|m,\ldots, m)$ is one-dimensional. Using
commutation relations for the matrix units $e_{ij}$ we infer that odd
superderivations $_{ij}D$ for $i<j$ supercommute. 

Therefore, applying PBW
theorem with the ordering of the generators of $Dist(G)$ as in Section 1  and 
using the fact the $v$ is the highest vector of $H^0_G(\lambda)$, we
conclude that the weight space $H^0_G(\lambda)_{\kappa}$ is generated by $(v)_1D\ldots _mD$.
Thus $(v)_{m}D \ldots _{1}D= 0$ implies that $L_G(\lambda)$ is a proper
submodule of $H^0_G(\lambda)$. Therefore $\lambda$ must be typical if $%
H^0_G(\lambda)$ is irreducible.

Assume now that $H_{G_{ev}}^{0}(\lambda )$ is an irreducible $G_{ev}$-module
and $\lambda $ is typical. Denote by $\mathcal{D}^+$ the subalgebra of $%
Dist(G)$ generated by even $e_{ij}^{(r)}$ and odd $e_{ij}$ for $i>j$ and $r>0
$; and by $\mathcal{D}_{ev}^+$ the subalgebra of $Dist(G)$ generated by even 
$e_{ij}^{(r)}$ for $i>j$ and $r>0$. Since $H^0_{G_{ev}}(\lambda)$ is
irreducible and its the highest weight space is one-dimensional and
generated by $v$, it is generated by $v$ as a $\mathcal{D}_{ev}^+$-module.
Every element $x$ of $\mathcal{D}^+$ can be be written as $x=\sum_{S\subset
M} e_Sx_S,$ where $M=\{1,\dots,m\}\times\{m+1,\ldots,m+n\}$,$%
e_S=\prod_{(i,j)\in S} e_{ij}$ and $x_S\in \mathcal{D}_{ev}^+$. We will show
that if $xv=0$, then $x_Sv=0$ for each $S$ as above. This implies that $%
H^0_G(\lambda)$ is generated by $v$ as a $\mathcal{D}^+$-module and the
character of the simple $G$-module generated by $v$ coincides with the
character of $H^0_G(\lambda)$, and consequently $H^0_G(\lambda)$ is
irreducible.

We proceed by induction and assume that we have already showed that $x_Tv=0$
for all $T\subsetneq S$. The commutation relations for the matrix units $%
e_{ij}$ of the Lie superalgebra $\mathfrak{gl}(m|n)$ imply that $%
e_{M\setminus S}e_S = \pm e_M$ and $[x_S,e_M]=0$. Using that and the
inductive assumption, we get 
\begin{equation*}
0=e_{M\setminus S}xv= e_{M\setminus S}\sum_{T\subset S} e_Tx_Tv=\pm
e_Mx_Sv=\pm x_S(e_Mv).
\end{equation*}
Using Proposition \ref{12} we compute that 
\begin{equation*}
x_S(e_Mv)= (\prod_{i=1}^m \omega_{i}(\lambda)) x_S(v y_{m}\ldots y_{1}).
\end{equation*}
Since $x_S(v y_{m}\ldots y_{1})=(x_S v) y_{m}\ldots y_{1}$ by Lemma \ref{13}
and $\lambda$ is typical, we infer that $x_Sv=0$ for each $S\subset M$.
\end{proof}

Recall the definition of the Weyl module 
$V_G(\lambda)= Dist(G)\otimes_{Dist(P)}V_{G_{ev}}(\lambda)$.

\begin{cor}
The Weyl module $V_G(\lambda)$ is irreducible if and only if $%
V_{G_{ev}}(\lambda)$ is irreducible and $\lambda$ is typical.
\end{cor}

\begin{proof}
Apply the contravariant duality given by supertransposition $\tau$ to get $%
V_G(\lambda)\cong H^0_G(\lambda)^{\langle \tau \rangle}$ and apply Theorem 1.
\end{proof}

If we further restrict to Weyl modules over general Lie superalgebras and
consider the case when the characteristic of $K$ equals zero, then we
recover the classical result of Kac.

Define the Kac module $K_G(\lambda)=Dist(G)\otimes_{Dist(P)}
L_{G_{ev}}(\lambda)$, where $L_{G_{ev}}(\lambda)$ is the simple $G_{ev}$%
-module of the highest weight $\lambda$, considered as $P$-supermodule 
via the epimorphism $P\to G_{ev}$.  Then $V_G(\lambda)$ has a
filtration by Kac modules induced by the filtration of $V_{G_{ev}}(\lambda)$
by simple $G_{ev}$-modules.

\begin{cor}
The Kac module $K_G(\lambda)$ is irreducible if and only if $\lambda$ is
typical.
\end{cor}

\end{document}